\documentclass[a4paper,12pt]{amsart}
\usepackage[dvipdfmx]{graphicx}
\usepackage{amsmath,amssymb,amsfonts,amsthm,color}
\newtheorem{definition}{Definition}[section]
\newtheorem{theorem}[definition]{Theorem}
\newtheorem{lemma}[definition]{Lemma}
\newtheorem{corollary}[definition]{Corollary}
\newtheorem{remark}[definition]{Remark}
\newtheorem{example}[definition]{Example}
\newtheorem{proposition}[definition]{Proposition}

\newcommand{\A}{\mathcal{A}}
\newcommand{\hilb}{\mathcal{H}}

\newcommand{\M}{\mathcal{M}}

\newcommand{\unit}{{\bf 1}}

\begin{document}

\title{Non-linear monotone positive maps}

\author{Masaru Nagisa}
\address[Masaru Nagisa]{Graduate School of Science, Chiba University, 
Chiba, 263-8522,  Japan}
\email{nagisa@math.s.chiba-u.ac.jp}
\author{Yasuo Watatani}
\address[Yasuo Watatani]{Department of Mathematical Sciences,
Kyushu University, Motooka, Fukuoka, 819-0395, Japan}
\email{watatani@math.kyushu-u.ac.jp}

\maketitle

\begin{abstract}
We study several classes of general non-linear positive maps between $C^*$-algebras, 
which are not necessary completely positive maps.  We characterize the class of the  compositions of *-multiplicative maps and positive linear maps
as the class of non-linear maps of boundedly positive type abstractly. 
We consider three classes of non-linear positive maps defined only on the positive cones, 
which are the classes of being monotone, supercongruent or concave. Any concave maps 
are monotone.  The intersection of the monotone maps and the supercongruent maps  
characterizes the class of monotone Borel functional calculus. We give many examples 
of non-linear positive maps, which show that there exist no other relations among these three classes in general.  
 
\end{abstract}

\section{Introduction}
We study several classes of general non-linear positive maps between $C^*$-algebras. 
Ando-Choi \cite{A-C} and Arveson \cite{Ar2} investigated non-linear completely 
positive maps and extend the Stinespring dilation theorem.   Ando-Choi showed that 
any non-linear completely positive map is decomposed as a doubly infinite sum of compressions of completely positive linear maps on certain $C^*$-tensor products. 
Arveson obtained the similar expression for bounded completely positive 
complex-valued functions on the open unit ball of a unital $C^*$-algebra. 
Hiai-Nakamura \cite{H-N} studied a non-linear counterpart of Arveson's 
Hahn-Banach type extension theorem \cite {Ar1} for completely positive linear maps. 
Beltita-Neeb \cite{B-N} studied non-linear completely positive maps and dilation 
theorems for real involutive algebras.  Recently Dadkhah-Moslehian \cite{D-M} studied 
some properties of non-linear positive maps like Lieb maps and the multiplicative 
domain for 3-positive maps.

We study general non-linear positive maps between $C^*$-algebras, 
which are not necessary completely positive maps.  First we study a non-completely 
positive variation of Stinespring type dilation theorem.  Let $A$ and $B$ be $C^*$-algebras.  
We consider non-linear positive maps $\varphi : A \rightarrow B$.  For instance, *-multiplicative maps , positive linear maps and their compositions are 
typical examples of non-linear positive maps. We characterize the 
class of the  compositions of these algebraically simple 
maps as non-linear maps of boundedly positive type abstractly. This class is different 
with the class of non-linear completely positive maps, because the transpose map of the 
$n$ by $n$ matrix algebra for $n \geq 2$ is contained in the class.  
They are not necessarily real analytic.

Another typical example of non-linear posiive mas is given as the 
functional calculus by a continuous positive function.  See, for example, 
\cite{bhatia1} , \cite{bhatia2} and \cite{Si}. 
In particular 
operator monotone functions are important to study operator means 
in Kubo-Ando theory in \cite{kuboando}.  Osaka-Silvestrov-Tomiyama 
\cite{O-S-T} studied monotone operator functions on $C^*$-algebras. 
Recently Hansen-Moslehian-Najafi \cite{hansenmn}
characterize the continuous functional calculus by a operator convex 
function by being of Jensen-type.  Moreover a sufficient condition is given by 
Anjidani \cite{An}.

We consider three classes of non-linear positive maps defined only on the positive cones, 
which are the classes of being monotone, supercongruent or concave. 
Let $A$ be a $C^*$-algebra. We denote by $A^+$ be the cone of all positive elements. 
A non-linear positive map $\varphi : A^+ \rightarrow B^+$ between 
$C^*$-algebras $A$ and $B$  is 
said to be {\it monotone } if for any $x, y \in A^+$, $x \leq y$ implies that 
$\varphi(x) \leq \varphi(y)$. 
We say $\varphi : A^+ \rightarrow B^+$ is {\it supercongruent}  
if $c\varphi(a)c \le \varphi(cac)$ 
for any $a\in A^+$ and any contraction $c\in  A^+$. 
A positive map $\varphi : A^+ \rightarrow B^+$ is 
said to be {\it concave}
 if $\varphi (tx + (1-t)y) \geq t\varphi (x) + (1-t)\varphi (x)$ 
 for any $x, y \in A^+$ and $t \in [0,1]$.

Let  $f: [0,\infty) \rightarrow [0,\infty)$ 
be a operator monotone {\it continuous} function , $H$ a Hilbert space and 
$\varphi_f  : B(H)^+ \rightarrow B(H)^+$ be a 
continuous functional calculus by $f$ denoted by $\varphi_f(a) = f(a)$  for 
$a \in B(H)^+$. Then $\varphi_f$ is a monotone, supercongruent, concave and normal 
positive map.

Let $M$ be a von Neumann algebra on a Hilbert space $H$ and  
$\varphi : M^+ \rightarrow M^+$ be the non-linear positive map defined by the 
 $\varphi(a) ={\text {(the range projection of a)}}$ for $a \in M^+$. Then $\varphi$ 
 is monotone, supercongruent and normal .  In fact, this map is a functional calculus 
 of $a$ by a  Borel function $\chi_{(0,\infty)}$ on $[0,\infty )$.

In this paper we shall show that any concave maps are monotone.  
The intersection of the monotone maps and the supercongruent maps  
characterizes the class of monotone Borel functional calculus. We give many examples 
of non-linear positive maps, which show that there exist no other relations among these three classes.

We also discuss  the ambiguity of operator means for non-invertible 
positive operators related with our Theorem. 
Based on the theory of Grassmann manifolds, Bonnabel-Sepulchre  \cite{B-S} 
and Batzies-H\''{u}per-Machado-Leite \cite{B-H-M-L} introduced  the 
geometric mean for positive semidefinite matrices or projections of fixed rank. 
Fujii \cite{F} extends it to a general theory of means of positive semideinite 
matrices of fixed rank.

Noncommutative function theory is important and related to our paper.  But the domain 
of a noncommutative function is graded, which is different with our simple one domain setting. Therefore we do not disscuss a relation with them here. It will be discussed 
in the future.

Finally we show a matrix version of the Choquet integral \cite{Ch}, 
the Sugeno integral   \cite{Su} 
or more generally  the inclusion-exclusionintegral by Honda-Okazaki \cite{H-O} 
for non-addiive monotone measures 
as another type of examples of non-linear monotone positive maps. 

This work was supported by JSPS KAKENHI Grant Number JP17K18739.

\section{Non-linear maps of boundedly positive type}
  Let $A$ and $B$ be  $C^*$-algebras.  We consider non-linear positive maps 
  $\varphi : A \rightarrow B$.  For instance, *-multiplicative maps , positive linear maps and their compositions are typical examples of non-linear positive maps. 
In this section, we characterize the class of the  compositions of these algebraically simple 
maps as non-linear maps of boundedly positive type abstractly. This class is different 
with the class of non-linear completely positive maps, because the transpose map of the 
$n$ by $n$ matrix algebra for $n \geq 2$ is contained in the class.  
They are not necessarily real analytic. 

\begin{definition} \rm 

Let $A$ and $B$ be  $C^*$-algebras.  A map $\varphi : A \rightarrow B$ is said to be 
of positive type if for any finite subset $\{a_1,a_2,\dots , a_n\} \subset  A$ and 
any finite subset $\{\alpha_1,\alpha_2,\dots , \alpha_n\}\subset {\mathbb C} $
$$
0 \leq \sum_{i=1}^{n}  \sum_{j=1}^{n} 
\overline{\alpha_i}{\alpha_j} \varphi (a_i^*a_j) . 
$$
 A map $\varphi : A \rightarrow B$ is said to be 
of boundedly positive type if for any $a \in A$, there exists a constant $K=K_a > 0$ such that 
 for any finite subset $\{a_1,a_2,\dots , a_n\} \subset  A$ and 
any finite subset $\{\alpha_1,\alpha_2,\dots , \alpha_n\}\subset {\mathbb C} $
$$
0 \leq \sum_{i=1}^{n}  \sum_{j=1}^{n} 
\overline{\alpha_i}{\alpha_j} \varphi (a_i^*a^*aa_j) 
 \leq K  
  \sum_{i=1}^{n}  \sum_{j=1}^{n} 
\overline{\alpha_i}{\alpha_j} \varphi (a_i^*a_j) .
$$
Recall that a map $\varphi : A \rightarrow B$ is said to be positive if 
for any $a \in A$  $0 \leq \varphi (a^*a)$. 
Assume that $A$ is unital. Then it is clear that if $\varphi : A \rightarrow B$ is 
of boundedly positive type, then $\varphi$ is  of positive type. If  $\varphi$ is 
of positive type, then $\varphi$ is  positive. 
\end{definition}

\begin{example} \rm
Let $A$ and $B$ be  $C^*$-algebras. If a map $\varphi : A \rightarrow B$ is 
a positive linear map, then 
$\varphi$ is of boundedly positive type.  In fact,  
for any non-zero  $a \in A$, put $K  = \|a\|^2> 0$ . Then 
 for any finite subset $\{a_1,a_2,\dots , a_n\} \subset  A$ and 
any finite subset $\{\alpha_1,\alpha_2,\dots , \alpha_n\}\subset {\mathbb C} $
\begin{align*}
0 \leq \sum_{i=1}^{n}  \sum_{j=1}^{n} 
\overline{\alpha_i}{\alpha_j} \varphi (a_i^*a^*aa_j) 
 & = \varphi ((\sum_{i=1}^{n}{\alpha_i}a_i) ^* a^*a
 (\sum_{j=1}^{n}{\alpha_j}a_j))\\
 & \leq \|a\|^2  
  \sum_{i=1}^{n}  \sum_{j=1}^{n} 
\overline{\alpha_i}{\alpha_j} \varphi (a_i^*a_j) .
\end{align*}
We may put $K = 1$ if $ a = 0$.  
\end{example}

\begin{example} \rm
Let $A$ and $B$ be  $C^*$-algebras. If a map $\varphi : A \rightarrow B$ is *-multiplicative, 
that is, $\varphi(ab) = \varphi(a)\varphi(b)$ and$\varphi(a^*) = \varphi(a)^*$ for any 
$a,b \in A$, then $\varphi$ is of boundedly positive type.  In fact,  
for any  $a \in A$, put $K  = \|\varphi(a)\|^2 + 1> 0$ . Then 
 for any finite subset $\{a_1,a_2,\dots , a_n\} \subset  A$ and 
any finite subset $\{\alpha_1,\alpha_2,\dots , \alpha_n\}\subset {\mathbb C} $
\begin{align*}
0 \leq \sum_{i=1}^{n}  \sum_{j=1}^{n} 
\overline{\alpha_i}{\alpha_j} \varphi (a_i^*a^*aa_j) 
 & = (\sum_{i=1}^{n} {\alpha_i}\varphi(a_i)) ^*{\varphi(a)}^*{\varphi(a)}
 (\sum_{j=1}^{n} {\alpha_j}\varphi(a_j))    \\
 & \leq (\|\varphi(a)\|^2 + 1)
  \sum_{i=1}^{n}  \sum_{j=1}^{n} 
\overline{\alpha_i}{\alpha_j} \varphi (a_i^*a_j) .
\end{align*}
For example the determinant  $det: M_n({\mathbb C}) \rightarrow {\mathbb C}$ 
is of boundedly positive type.  Let $B = A \otimes_{min} \dots \otimes_{min} A$ and 
$\varphi : A \rightarrow B$ be defined by $\varphi (a) = a \otimes \dots \otimes a$, 
then $\varphi$ is of boundedly positive type. 
\end{example}

We shall study the class of maps of  boundedly positive type.  Let  
$A$, $B$ and $C$ be  unital $C^*$-algebras. If $\varphi_1 : A \rightarrow C$  is 
*-multiplicative and $\varphi_2 : C \rightarrow B$  is a positive linear map, 
then the composition $\varphi = \varphi_2 \circ \varphi_1$ is of boundedly positive type. 
Conversely any map of boundedly positive type is of this form. 

\begin{theorem}
Let $A$ and $B$ be  unital $C^*$-algebras. Consider a map $\varphi: A \rightarrow B$ . 
Then the following are equivalent: \\
\begin{enumerate}
\item[$(1)$]  $\varphi$ is of boundedly positive type. 
\item[$(2)$] There exists a unital $C^*$-algebra $C$, 
 a *-multiplicative map $\varphi_1 : A \rightarrow C$  
 and  a positive linear map $\varphi_2 : C \rightarrow B$ such that $\varphi$ is 
 the composition $\varphi = \varphi_2 \circ \varphi_1$ of these maps. 
\end{enumerate}
\end{theorem}
\begin{proof}
(2) $\Rightarrow$ (1): Assume (2).  
For any  $a \in A$, put $K  = \|\varphi_1(a)\|^2 + 1> 0$ . Then 
 for any finite subset $\{a_1,a_2,\dots , a_n\} \subset  A$ and 
any finite subset $\{\alpha_1,\alpha_2,\dots , \alpha_n\}\subset {\mathbb C} $
\begin{align*}
0 \leq \sum_{i=1}^{n}  \sum_{j=1}^{n} 
\overline{\alpha_i}{\alpha_j} \varphi (a_i^*a^*aa_j) 
 & =\varphi_2( (\sum_{i=1}^{n} 
 {\alpha_i}\varphi_1(a_i)) ^*{\varphi_1(a)}^*{\varphi_1(a)}
 (\sum_{j=1}^{n} {\alpha_j}\varphi_1(a_j)) )   \\
 & =\leq (\|\varphi_1(a)\|^2 + 1)\varphi_2( (\sum_{i=1}^{n} 
 {\alpha_i}\varphi_1(a_i)) ^*
 (\sum_{j=1}^{n} {\alpha_j}\varphi_1(a_j)) )   \\
& \leq (\|\varphi_1(a)\|^2 + 1)
  \sum_{i=1}^{n}  \sum_{j=1}^{n} 
\overline{\alpha_i}{\alpha_j} \varphi (a_i^*a_j) .
\end{align*}
(1) $\Rightarrow$ (2): Assume (1).  Let ${\mathcal S}_A$ be a *-semigroup defined 
by  ${\mathcal S}_A = A$ as a set with the product $ab$ and the involution $a^*$ 
borrowed from $A$. Consider an algebraic  
*-semigroup algebra ${\mathbb C}[{\mathcal S}_A] $
of ${\mathcal S}_A$  with linear basis $\{u_a \ | \ a \in A \}$ :
$$
{\mathbb C}[{\mathcal S}_A] 
= \{ x = \sum_i  x_iu_{a_i } \  |  \ x_i \in {\mathbb C}, a_i \in A \}
$$
with the product $u_a u_b = u_{ab}$ and the involution $(u_a)^* = u_{a^*}$ 
for $a,b \in A$ . Thus  ${\mathbb C}[{\mathcal S}_A] $ is a algebraic *-algebra. 
Define a linear map $\psi : {\mathbb C}[{\mathcal S}_A]  \rightarrow B$  by 
$$
\psi(\sum_i  x_iu_{a_i } ) = \sum_i  x_i\varphi({a_i }). 
$$
For  any state $\omega$ on $B$, define a linear functional $\psi_{\omega}$ on   
${\mathbb C}[{\mathcal S}_A] $ by $\psi_ {\omega} = \omega \circ \psi $. 
We introduce a pre-inner product $<,>_{\omega}$ on ${\mathbb C}[{\mathcal S}_A] $ 
by 
$$
<x,y>_{\omega} := \psi_ {\omega}(y^*x) 
= \omega(\sum_i \sum_j 
\overline{y_j}{x_i} \varphi (b_j^*a_i))
$$
for $x = \sum_i  x_iu_{a_i },\  y= \sum_i  y_ju_{b_j} \in 
{\mathbb C}[{\mathcal S}_A] $,  $(x_i, y_j \in {\mathbb C}, a_i , b_j \in A ).$
Then 
$$<x,x>_{\omega}  = \psi_ {\omega}(x^*x) 
= \omega(\sum_i  \sum_j 
\overline{x_j}{x_i} \varphi (a_j^*a_i)) \geq 0, 
$$
since $\varphi$ is of positive type .  Let 
$N_{\omega} := \{ x   \in {\mathbb C}[{\mathcal S}_A]  \ | \ <x,x>_{\omega} =0 \}$. 
Define a Hilbert space $H_{\omega}$ by the completion of  
${\mathbb C}[{\mathcal S}_A]/ N_{\omega}$. 
 Let $\eta_{\omega} : {\mathbb C}[{\mathcal S}_A] \rightarrow  H_{\omega}$ 
 be the canonical map such that 
 $<\eta_{\omega}(x), \eta_{\omega}(y)> = <x,y>_{\omega}$.  
 Then we have a *-representation 
 $\pi_{\omega} : {\mathbb C}[{\mathcal S}_A] \rightarrow  B(H_{\omega})$ such that
 $\pi_{\omega}(u_a)\eta_{\omega}(x) = \eta_{\omega}(u_ax)$ and 
 $\pi_{\omega}(u_a)^* = \pi_{\omega}(u_a^*) = \pi_{\omega}(u_a^*)$ 
 for $a \in A$ .  In fact
 
\begin{align*}
 \|\eta_{\omega}(u_ax)\|^2 
&= <\sum_i  x_iu_{aa_i }, \sum_j  x_ju_{aa_j}>_{\omega}\\
&=\omega(\sum_i  \sum_j 
\overline{x_j}{x_i} \varphi (a_j^*a^*aa_i))\\
&\leq K\omega (\sum_i  \sum_j 
\overline{x_j}{x_i} \varphi (a_j^*a_i))  = K  \|\eta_{\omega}(x)\|^2
\end{align*}
since $\varphi$ is of boundedly positive type, where $K$ depends only on $a$. 
Therefore $\pi_{\omega}(u_a)$ is a well defined bounded operator with 
$\|\pi_{\omega}(u_a)\| \leq \sqrt{K}$.  Since $A$ has a unit $I$, we have that 
$$
<\pi_{\omega}(u_a)\eta_{\omega}(u_I), \eta_{\omega}(u_I)> 
=  <\eta_{\omega}(u_au_I), \eta_{\omega}(u_I)> 
={\omega}(\psi(u_a)) = {\omega}(\varphi(a))
$$
Moreover for $x = \sum_i  x_iu_{a_i } \in {\mathbb C}[{\mathcal S}_A] $ , we have that
$$
{\omega}(\psi(x)) = <\pi_{\omega}(x)\eta_{\omega}(u_I), \eta_{\omega}(u_I)> 
$$
Next we shall consider a $\varphi$-universal representation $\pi_u$  of  a *-algebra 
${\mathbb C}[{\mathcal S}_A]$  on a Hilbert space $H_u$ as follows: 
Put 
$$
H_u = \oplus \{H_{\omega} \ | \ \omega \text{ is a state on } B \}
$$
and 
$$
\pi_u = \oplus \{\pi_{\omega} \ | \ \omega \text{ is a state on } B \}.
$$
For $x   \in {\mathbb C}[{\mathcal S}_A]$ define the $\varphi$-universal seminorm 
$$
\|x\|_u := \sup \{ \|\pi_{\omega}(x)\| \  | \ \omega \text{ is a state on } B \}
\leq \sum_i |x_i|\sqrt{K_{a_i}} < \infty.
$$
Let $C$ be the completion of ${\mathbb C}[{\mathcal S}_A]/ {\rm Ker }  \pi_u$ by 
the induced norm $\|[x]\|_u$ . Then $C$ is a $C^*$-algebra and isomorphic to 
the closure of $\pi_u({\mathbb C}[{\mathcal S}_A])$.  We also have a 
*-representaion $\overline{\pi_u}$ of $C$ on $H_u$  such that 
$\overline{\pi_u}([x]) =  \pi_u(x)$. \\
Next we shall show that for $x \in {\mathbb C}[{\mathcal S}_A]$
$$
\| \psi (x) \| \leq 2\|\varphi(I)\| \|x\|_u.
$$
In fact, since 
\begin{align*}
|{\omega}(\psi(x)) 
&|= |<\pi_{\omega}(x)\eta_{\omega}(u_I), \eta_{\omega}(u_I)> |\\
&\leq \| \ \pi_{\omega}(x) \| \|\eta_{\omega}(u_I)\|^2
= \| \ \pi_{\omega}(x) \| {\omega}(\varphi (I)) 
\leq \| \varphi (I) \|  \| \ \pi_u(x) \|.
\end{align*}
we have that 
$$
\| \psi (x) \| \leq 2(\text{ numerical radius of }(\psi (x)) 
\leq 2\|\varphi(I)\| \|\pi_u(x)\|
\leq 2\|\varphi(I)\| \| x \|_u
$$
Therefore ${\rm Ker }  \pi_u \subset {\rm Ker }  \psi$. Hence there exists a linear map 
$\tilde{\psi} : {\mathbb C}[{\mathcal S}_A]/ {\rm Ker }  \pi_u \rightarrow B$ 
such that $\tilde{\psi}([x]) = \psi (x)$. Moreover $\tilde{\psi} $ extends to 
a linear map $\varphi_2 : C \rightarrow B$ by the boundedness of $\tilde{\psi}$. 
Then $\varphi_2$  is a positive linear map.  In fact, for 
$x \in {\mathbb C}[{\mathcal S}_A]$,  since $\varphi$ is of positive type, 
$$
\varphi_2([x^*x]) = \psi(x^*x)= \psi((\sum_i  \sum_j 
\overline{x_j}{x_i} \varphi (a_j^*a_i)) \geq 0
$$
Any positive element in $C$  can be  approximated with these $[x^*x]$  for 
 $x \in {\mathbb C}[{\mathcal S}_A]/{\rm Ker }  \pi_u$. By the continuity of $\varphi_2$ , 
 $\varphi_2$  is also positive. \\
 We define $\varphi_1: A \rightarrow C$ by $\varphi_1(a) = [u_a]$. 
 Then $\varphi_1$ is *-multiplicative.  Moreover  
 $$
 \varphi_2 \circ  \varphi_1(a) = \varphi_2([u_a]) = \psi (u_a) = \varphi(a). 
 $$
 for $a \in A.$

\end{proof}
\begin{definition} \rm
Let $A$ and $B$ be  $C^*$-algebras.  For a map $\varphi : A \rightarrow B$ and 
a natural number $n$,  $\varphi_n: M_n(A) \rightarrow M_n(B)$ is defined by 
$\varphi_n((a_{ij})_{ij})= (\varphi(a_{ij}))_{ij}$. Then $\varphi$ is 
completely positive if $\varphi_n$ is positive for any $n$.  $\varphi$ is said to be 
positive definite if, for any $n$ and for any $\{a_1,a_2,\dots , a_n\} \subset  A$, 
$\varphi(a_i^*a_j))_{ij}\in  M_n(B)$ is positive as in  \cite[Definition 2.8] {B-N}   . 
\end{definition}
\begin{remark}  \rm 
(1)Let $A$ and $B$ be  $C^*$-algebras. If a map $\varphi: A \rightarrow B$  is a non-linear map. If $\varphi$ is completely positive,  then $\varphi$ is positive definite. 
If $\varphi$ is positive definite,then $\varphi$ is of positive type.  But the converses 
do not hold.  In fact, the transpose map of the 
$n$ by $n$ matrix algebra for $n \geq 2$ is of positive type but is not positive definite. \\
(2)We should note that the class of completely positive maps and the class of 
maps of  boundedly positive type are different.  For example, let 
$A = B = {\mathbb C}$ and 
$\varphi(z) = e^z$. Then $\varphi$ is completely positive but is not of  boundedly positive type. In fact, there exist no constant $K > 0$ such that for any $z \in {\mathbb C}$, 
$\varphi(z^*3^2z) \leq K \varphi(z^*z)$. The transpose map of the 
$n \times n$ matrix algebra for $n \geq 2$ is of boundedly positive type but is
 not completely positive. 
\end{remark}

\section{Some classes of non-linear positive maps defined only on the positive cones}

Let $A$ be a $C^*$-algebra. We denote by $A^+$ be the cone of all positive elements. In this section we consider non-linear posiive maps defined only on the positive cones. 

\begin{definition} \rm 
Let $A$ and $B$ be  $C^*$-algebras. 
A non-linear positive map $\varphi : A^+ \rightarrow B^+$ is 
said to be {\it monotone } if for any $x, y \in A^+$, $x \leq y$ implies that 
$\varphi(x) \leq \varphi(y)$. 
We say $\varphi : A^+ \rightarrow B^+$ is {\it supercongruent}  
if $c\varphi(a)c \le \varphi(cac)$ 
for any $a\in A^+$ and any contraction $c\in  A^+$. 
A positive map $\varphi : A^+ \rightarrow B^+$ is 
said to be {\it concave}
 if $\varphi (tx + (1-t)y) \geq t\varphi (x) + (1-t)\varphi (x)$ 
 for any $x, y \in A^+$ and $t \in [0,1]$.

 When $A$ and $B$ are von Neumann algebras, $\varphi : A^+ \rightarrow B^+$ is 
 said to be {\it normal} if , 
for any bounded increasing net $a_{\nu} \in A^+$,
\[   \varphi(\sup_{\nu} a_{\nu}) = \sup_{\nu} \varphi(a_{\nu}) . \]
\end{definition}

\begin{example} \rm  
Let $f: [0,\infty) \rightarrow [0,\infty)$ 
be a operator monotone continuous function , $H$ a Hilbert space and 
$\varphi_f  : B(H)^+ \rightarrow B(H)^+$ be a 
continuous functional calculus by $f$ denoted by $\varphi_f(a) = f(a)$  for 
$a \in B(H)^+$. Then $\varphi_f$ is a monotone, supercongruent, concave and normal 
positive map. 
\label{monotonesupercongruent}
\end{example}



There exists a non-linear positive map $\varphi = : B(H)^+ \rightarrow B(H)^+$
which is monotone, supercongruent and normal  but is not a continuous functional calculus.  For example, let $\varphi(a)$ be the projection onto the closure of the range 
of $a \in B(H)^+$, then $\varphi(a)$ is called the range projection of $a$ or  
the supprot projection of $a$ and is equal to 
the projection onto the orthogonal complement of
the kernel of $a$ (\cite[2.22]{stratila}). 

\begin{proposition}
Let $M$ be a von Neumann algebra on a Hilbert space $H$ and  
$\varphi : M^+ \rightarrow M^+$ be the non-linear positive map defined by the 
 $\varphi(a) ={\text {(the range projection of a)}}$ for $a \in M^+$. Then $\varphi$ is 
 is monotone, supercongruent and normal .  
 \label{rangeprojection}
\end{proposition}

\begin{proof}
For $a,b\in M^+$, we remark the following facts:
\begin{itemize}
  \item $\varphi(a) \le \varphi(b)$ is equivalent to ${\rm Ker }(a) \supset {\rm Ker }(b)$.
  \item ${\rm Ker}(a) = {\rm Ker}(\varphi(a))$.
  \item $\varphi(a)\ge a$  if  $\|a\|\le 1$.
  \item $\varphi(a)=a$ if $a$ is a projection.
\end{itemize}

For $0\le a \le b$, it is clear that ${\rm Ker }(a) \supset {\rm Ker }(b)$.
So $\varphi$ is monotone.

For any contraction $c$, we have
\[  c^*ac\xi = 0 \Rightarrow a^{1/2}c \xi = 0 \Rightarrow \varphi(a)c\xi =0 \Rightarrow c^*\varphi(a)c\xi =0. \]
Since ${\rm Ker }(c^*ac) \subset {\rm Ker }(c^*\varphi(a)c)$,  
$\varphi(c^*\varphi(a)c) \le \varphi(c^*ac)$.  
Because $c^* \varphi(a)c$ is a contraction, 
$$
c^*\varphi(a)c\le \varphi(c^*\varphi(a)c) \le \varphi(c^*ac)
$$.
So $\varphi$ is supercongruent.

Let $\{a_\nu\}$ be a bounded increasing net in $ M^+$.
Since
\[  {\rm Ker}(\sup_\nu \varphi(a_\nu)) = \bigcap_\nu {\rm Ker}(\varphi(a_\nu))
    = \bigcap_\nu {\rm Ker}(a_\nu) = {\rm Ker}(\sup_\nu a_\nu) \]
and $\sup_\nu \varphi(a_\nu)$ is a projection, we have
\[  \varphi(\sup_\nu a_\nu) = \varphi (\sup_\nu \varphi(a_\nu)) = \sup_\nu \varphi(a_\nu).  \]
So $\varphi$ is normal on $ M^+$.
\end{proof}

We shall  study and compair these properties of being monotone, supercongruent and 
concave for  general non-linear positive maps $\varphi : A^+ \rightarrow B^+$ on 
the whole positive cone  $ A^+$   of  a $C^*$-algebra $A$.  
If $\varphi$ is concave, then $\varphi$ is monotone. But there exist no other relations 
between them in general as follows:

\begin{proposition}
Let $A$ and $B$ be  $C^*$-algebras and 
$\varphi : A^+ \rightarrow B^+$  be 
a non-linear positive map. 
If $\varphi$ is concave, then $\varphi$ is monotone.
\end{proposition}
\begin{proof}  Assume that $\varphi$ is concave. 
For $0\le a \le b$, we define $\{a_{k}\}$ as follows:
\[ a_{0}=a, \; a_{1}=b, a_{k} = a + k(b-a) \geq  0 \ \  k=0,1,2,\ldots  \]
Then $a_{k+1}=\frac{a_{k}+a_{k+2}}{2}$.  
By the cancavity of $\varphi$, it follows
\[   \varphi(a_{k+1}) \ge \frac{1}{2}(\varphi(a_{k})+ \varphi(a_{k+2}) ) .  \]
So we have
\[  \varphi(a_{k+2})-\varphi(a_{k+1}) \le \varphi(a_{k+1}) -\varphi(a_{k}), \; k=0,1,2,\ldots \]
and
\[  \varphi(a_{n}) = \varphi(a_{0}) + \sum_{k=1}^{n} (\varphi(a_{k})-\varphi(a_{k-1}) )
     \le \varphi(a_{0}) + n (\varphi(a_{1})-\varphi(a_{0}) ). \]
We may assume that $A \subset B(H)$ for some Hilbert space $H$. 
We shall show that  $\varphi(a) \le \varphi(b)$. 
On the contrary, suppose it were not so. 
Then there exists a vector $\xi \in H$ such that 
$ \langle \varphi(a)\xi | \xi \rangle > \langle \varphi(b)\xi | \xi \rangle $. 
That is, $ \langle(\varphi(a_1)-\varphi(a_0))\xi | \xi \rangle <0). $
then $\langle \varphi(a_{n})\xi | \xi \rangle <0$ for a sufficiently large $n$.
This contradicts to that $\varphi(a_n) \ge 0$ for any $n$.  Hence we have that 
$\varphi(a) \le \varphi(b)$.
\end{proof}

\begin{proposition}
There exist many non-linear positive maps  on the positive cones of some 
$C^*$-algebras which satisfy anyone of the following  conditions:
\begin{enumerate}
  \item[$(1)$] $\varphi$ is concave and not supercongruent.
  \item[$(2)$] $\varphi$ is monotone, not concave and not supercongruent.
  \item[$(3)$] $\varphi$ is not monotone and supercongruent.
  \item[$(4)$] $\varphi$ is not monotone and not supercongruent.
  \end{enumerate}
\end{proposition}
\begin{proof} In each case, this was verified by construting many concrete examples  
in the below. 
\end{proof}
Remark. We shall show that if $\varphi$ is  monotone and supercongruent, then 
$\varphi$ is concave in the next section. \\

\noindent
Example (1-1)  let $M$ be a  ${\rm II}_1$-factor and $\tau$ the trace on $M$. 
Define $\varphi : M^+ \rightarrow M^+$ by $\varphi(a) = \tau( a )\unit$ for 
$a \in M^+ $.  Since  $\varphi$ is the restriction of a linear map, 
$\varphi$ is concave.
Let $p$ be a projention in $M$ with $\tau(p)=\frac{1}{2}$.
Since
\[  p \tau(\unit)p = p \nleq \tau(p\unit p)\unit = \tau(p) \unit =\frac{1}{2}\unit, \]
$\varphi$ is not supercongruent.

\noindent
Example (1-2) Let $H$ be a Hilbert space and $M=B(H)$. Condsider a projection  
$p(\neq \unit)$ of $B(H)$ and take a vector $\xi \in pH$ with $\|\xi\|=1$. 
Define $\varphi : M^+ \rightarrow M^+$ by 
$\varphi(a) = \langle a\xi, \xi \rangle \unit$ for $a \in B(H)^+$.  
Since $\varphi$ is the restriction of a linear kap,  $\varphi$ is concave. 
By the fact 
\[   (\unit-p)\varphi(p)(\unit -p) = \unit -p \nleq \varphi((\unit-p)p(\unit-p)=0, \]
$\varphi$ is not supercongruent.

\noindent
Example (2-1) Let  $H = \ell ^2(\mathbb N)$ and $M = B(H)$.  Consider 
a maximal abelian *-subalgebra $A \cong \ell^{\infty}(\mathbb N)$  
of $B(H)$ and a conditional expectation $E$ of $B(H)$ onto $A$. 
We define $\varphi(a ) = E(a)^{2}$ for $a \in B(H)^+$. 
Since $E$ is positive linear map and the mapping 
$\A^+ ni a \mapsto a^{2} \in A^+$ is monotone, $\varphi$ is monotone. 
By the fact
\begin{gather*}
  \frac{\varphi(0\unit)+\varphi(2\unit)}{2} =2\unit \nleq \unit =\varphi(\unit) = \varphi(\frac{0\unit +2\unit}{2}), \\
  \frac{\unit}{2}\varphi(\unit)\frac{\unit}{2}=\frac{\unit}{4} \nleq \frac{\unit}{16}=\varphi(\frac{\unit}{4})
  =\varphi(\frac{\unit}{2}\cdot \unit \cdot \frac{\unit}{2}),
\end{gather*}  
$\varphi$ is not concave and not superconvergent.

\noindent
Example (3-1) Let $H$ be a Hilbert space and $M=B(H)$.  For $a \in B(H)^+$, define 
$\varphi(a) = \begin{cases} \unit  & \|a\|\le 1  \\  a & \|a\|>1 \end{cases}$. \\
Let $p(\neq \unit)$ be a projection.
Then we have 
\[  \varphi(\frac{1}{2}p) = \unit \nleq \varphi(2p)=2p.   \]
So $\varphi$ is not monotone.

Let $c \in B(H)^*$ be a contraction.  If $\|a\|\le 1$, then $c^{*}\varphi(a) c=c^{*}c\le \unit = \varphi(c^{*}ac)$.
If $\|a\|>1$, then $\varphi(a)=a$ and
\begin{align*}
   \varphi(c^{*}ac) & = \begin{cases} c^{*}ac, \quad & \|c^{*}ac\| >1 \\
                                  \unit,  & \|c^{*}ac\|\le 1  \end{cases}  \\
                         &  \ge c^{*}\varphi(a)c .  
\end{align*}
So $\varphi$ is supercongruent.

\noindent
Example (3-2) Let $\M$ be a ${\rm II}_1$-factor.  
Define $\varphi : M^+ \rightarrow M^+$ by
$\varphi(a) = \begin{cases} \unit  &  a \text{ is invertible}  \\
                                  2 \unit & a \text{ is not invertible}  \end{cases} $ ,\\
for $a \in M^+$.  
It is clear that $\varphi$ is not monotone, since, for any non invertible positive contraction $a$,
\[ \varphi(a) =2 \unit \ge \unit =\varphi(\unit), \text{ and } a\le \unit.  \]
If $a$ is invertible, then
\[  c^*\varphi(a) c = c^{*}c \le \unit \le\varphi(c^{*}ac).  \]
If $a$ is not invertible, 
\[  c^{*}\varphi(a)c = 2 c^{*}c \le 2\unit =\varphi(c^{*}ac),  \]
where we use the fact that a left invertible element in a factor of type ${\rm II}_1$ is invertible.
So we have that $\varphi$ is supercongruence.

\noindent
Example (3-3)  Let $M$ be a ${\rm II}_1$-factor 
and $\tau:$ the normalized trace on $M$. Let 
$\alpha :[0,1]\longrightarrow [0,\infty)$ be  a decreasing and non-constant function. 
For $a \in M^+$, put $r(a)$ be the range projection of $a$. 
Define $\varphi : M^+ \rightarrow M^+$ by
\[    \varphi(a) = \alpha (\tau( r(a) ) ) \unit .  \]
By definition, there exist $t_{0}$, $t_{1}$ with  $0\le t_{0}<t_{1}$ and $\alpha(t_{0})>\alpha(t_{1})$.
We can chose projections $p,q$ with $\tau(p)=t_0$, $\tau_(q)=t_1$, and $p\le q$.
Then we have $\varphi(p) > \varphi(q)$.
So $\varphi$ is not monotone.

For $x\in M$, we denote $r(x)$ (resp. $s(x)$) the range projection of $x$ (resp. the support projection of $x$).
Let $c,x \in M^+$ with $\|c\|\le 1$. 
We set $p=r(x)$ and consider the polar decomposition of pc as follows:
\[  pc = hv ,  \]
where $h\ge 0$ and $r(v)=s(h)\le p$, $v^{*}v=s(pc)$, and $vv^{*}=s(h)$.
Since $M$ is a factor of type ${\rm II}_{1}$, there exists a unitary $u\in M$ satisfying $u^{*}s(h) = v^{*}$.
Then we have
\begin{align*}
  s(c^{*}xc) & = s(c^{*}pxpc) = s(v^{*}hxhv) = s(u^{*}hxhu) \\
    & = u^{*} s(hxh)u \le u^{*}s(p)u = u^{*}s(x) u.
\end{align*}
Since
\[  \tau(s(c^{*}xc))\le \tau(u^{*}s(x)u) = \tau(s(x)),  \]
we can prove the  supercongruence of $\varphi$ as follows:
\[  \varphi(c^{*}xc) = \alpha(\tau(c^{*}xc))\unit \ge \alpha(\tau(s(x)))\unit \ge c^{*}\alpha(\tau(s(x)))c
    = c^{*}\varphi(x) c. \]

\noindent
Example (3-4)  Let $H$ be a Hilbert space and $M=B(H)$.  For $a \in B(H)^+$, define
\[   \varphi(a) = \begin{cases} \unit, \quad & {\rm rank}(a)=\infty \\
                       2\unit  & {\rm rank}(a)<\infty  \end{cases} .  \]
Let $p$ be a finite rank projection.
By the fact $\varphi(p)=2 \unit >\unit = \varphi(\unit)$, $\varphi$ is not monotone.
If ${\rm rank}(a) <\infty$, then ${\rm rank}(c^*ac)<\infty$ and
\[  c^{*}\varphi(a)c =2c^{*}c \le 2 I = \varphi(c^{*}ac).\]
If ${\rm rank}(a) =\infty$, then 
\[  c^{*}\varphi(a)c =c^{*}c \le  I \leq  \varphi(c^{*}ac).\]
So $\varphi$ is supercongruent.

\noindent
Example (4-1)  Let $H$ be a Hilbert space and $M=B(H)$. 
For $a \in B(H)^+$, define $\varphi(a)= a^{2}$. 
Because $f(x) = x^2$ is not an operator monotone function, 
$\varphi$ is not monotone. 
\[  \frac{1}{2}\unit \cdot \varphi(\unit) \cdot \frac{1}{2}\unit = \frac{1}{4}\unit \nleq
    \varphi(\frac{1}{2}\unit \cdot \unit \cdot \frac{1}{2}\unit) = \varphi(\frac{1}{4}\unit) =\frac{1}{16}\unit. \]
This implies that $\varphi$ is not supercongruent.

\noindent
Example (4-2) Let $f$ be a real function $f(x) = 1 \vee x$.  
Let $H$ be a Hilbert space and $M=B(H)$. For $a \in B(H)^+$, define
$\varphi(a)= \unit \vee a = f(a)$ by a functional calculus. 
Consider
\begin{gather*}  a=\begin{pmatrix} 1 & 1 \\ 1 & 1 \end{pmatrix} \le 
    b= \begin{pmatrix} 3 & 0 \\ 0 & 3/2 \end{pmatrix}. \\
       \varphi(a) = \begin{pmatrix} 3/2 & 1/2 \\ 1/2 & 3/2 \end{pmatrix} \nleq
    \varphi(b) = \begin{pmatrix} 3 & 0 \\ 0 & 3/2 \end{pmatrix}. 
\end{gather*}
Thus $\varphi$ is not monotone.

Consider 

\begin{gather*}  a=\begin{pmatrix} 2 & 0 \\ 0 & 0 \end{pmatrix} , \quad
    c= \begin{pmatrix} 1/2 & 1/2\\ 1/2& 1 /2\end{pmatrix}. \\
    c\varphi(a)c = c\begin{pmatrix} 2 & 0 \\ 0 & 1 \end{pmatrix}c =\begin{pmatrix} 3/4 & 3/4 \\ 3/4 & 3/4 \end{pmatrix} \nleq
    \varphi(cac) = \begin{pmatrix} 1 & 0 \\ 0 & 1 \end{pmatrix}
\end{gather*}
Hence $\varphi$ is not supercongruent.

\section{Characterization of monotone maps given by  Borel functional calculus}
Let $M$ be a von Neumann algebra on a Hilbert space $H$ and  
$\varphi : M^+ \rightarrow M^+$ be the non-linear positive map defined by the 
range projection $\varphi(a)$ of $a \in M^+$. Then we showed that $\varphi$ 
is monotone, supercongruent and normal .  This is a typical example of non-linear 
 positive map which is  
 monotone, supercongruent and normal  but is not a form of continuous functional  
 calculus.  We should remark that this map is given by a  Borel functional culculus of the 
 Borel function $\chi_{(0,\infty)}$ on $[0,\infty )$ as follows:
\[   \varphi(a) = \chi_{(0,\infty)}(a), \]
where 
\[ \chi_{(0,\infty)} (t) = \begin{cases} 0 & t = 0  \\
1 & t>0  \end{cases} .\]
 In this section, we shall characterize monotone maps given by  Borel functional calculus. 
 At first we recall Borel functional calculus. 
 Let $\Omega$ be a metrizable topological space and $C(\Omega)$ a set of all complex valued continuous functions on $\Omega$.
We denote by $\mathcal{B}(\Omega)$ the set of all bounded complex Borel functions on $\Omega$. For a bound self-adjoint linear operator $a\in B(H)$ 
there exists a correspondence
\[   \mathcal{B}(\sigma(a)) \ni f \mapsto f(a) \in B(H)  \]
satisfying
\begin{itemize}
  \item[(1)] $f(a) = \alpha_0 \unit + \alpha_1 a + \cdots + \alpha_n a^n$ for any polynomial $f(\lambda) =\alpha_0 + \alpha_1\lambda + \cdots + \alpha_n\lambda^n$.
  \item[(2)] $(f_n)_n$ is a bounded sequence in $\mathcal{B}(\sigma(a))$.
If $(f_n)_n$ tends to $f\in \mathcal{B}(\sigma(a))$ with respect to the point-wise convergent topology, then the sequence $(f_n(a))_n$ of operators tends to the 
operator $f(a)$ in the strong operator topology.
  \item[(3)] If $f$ is continuous on $\sigma(a)$, then the Borel functional calculus 
  coincides with the continuous functional calculus. 
\end{itemize}
Moreover, this correspondence is a *-homomorphism of $\mathcal{B}(\sigma(a))$ onto the von Neumann algebra generated by $a$.
(see \cite{stratila}2.20.)
We call $f(a)$ the Borel functional calculus of $a$ by $f\in \mathcal{B}(\sigma(a))$.

The following fact is well-known (see, \cite[Theorem V.2.3]{bhatia1}).
%
%
\begin{lemma}
Let $f$ be an operator monotone {\it continuous} function on an interval $J$ 
which contains $0$ and $f(0)\ge 0$.
Then we have
\[   c^{*}f(a)c \le f(c^{*}ac),  \]
for  any $a=a^{*}\in B(H)$ with $\sigma(a)\subset J$ and
any $c\in B(H)$ with $\|c\|\le 1$.
\end{lemma}

%
%
\begin{theorem}
Let $M$ be an infinite-dimensional factor on a Hilbert space $H$ and 
$\varphi : M^+ \longrightarrow M^+$ be a non-linear positive map. 
Then the following are equivalent: 
\begin{enumerate}
\item[$(1)$]  $\varphi$ is monotone and supercongruent. 
\item[$(2)$] There exists a Borel function $f:[0,\infty) \longrightarrow [0,\infty)$ such that 
$f$ is continuous on $(0,\infty)$ , operator monotone on $(0,\infty)$ with
\[   f(0) \le \lim_{t\to 0+}f(t),   \]
and  $\varphi(a)$ is equal to the Borel functional calculus $f(a)$ of $a$ by $f$ 
for any $a\in M^+$.
\end{enumerate}
\label{Borelcharacterization}
\end{theorem}
\begin{proof}
(1) $\Rightarrow$ (2): Assume that $\varphi$ is monotone and supercongruent.  
Firstly, we shall show that,  for any $a \in M^+$ and any projection $p \in M$, if 
$ap = pa$, then  $p\varphi(a) = \varphi(a) p = p\varphi(pap)p$. 

In fact, suppose that $ap = pa$.  Then $pap=a^{1/2}pa^{1/2}\le a$. 
Since $\varphi$ is supercongruent and monotone, we have
\[  p \varphi(a) p \le \varphi (pap) \le \varphi(a).  \]
The positivity of $\varphi(a)-p\varphi(a)p$ implies $p\varphi(a)(\unit-p)=0$ and
$(\unit -p)\varphi(a)p=0$. 
So $p\varphi(a) = \varphi(a) p = p\varphi(pap)p$. 

Take $t\unit$ for any $t \in [0, \infty)$.  Because $p(t\unit) = (t\unit)p$, we have that 
$p\varphi(t\unit) = \varphi(t\unit)p$. Since $M$ is a factor, $\varphi(t\unit)$ is a 
scalar operator $f(t)\unit$.  Thus $f$ turns out to be a (not necessarily continuous) 
function $f:[0,\infty)\longrightarrow [0,\infty)$ such that 
$$
\varphi(t\unit) = f(t)\unit \ \ {\   for \ any \ } t \in [0, \infty). 
$$ 
By the monotonicity of $\varphi$, $f$ is increasing on $[0,\infty)$.
In particular, we have 
\[   f(0) \le \lim_{t\to 0+} f(t).   \]

Moreover for any $a \in M^+$ and any $\xi \in {\rm Ker }  a$, we have 
$$
\varphi(a)\xi = f(0)\xi 
$$
In fact, let $r$ be the range projection of $a$.  Then $\unit - r$ is the projection 
onto the  ${\rm Ker }  a$.  Since $ar = ra$,  we have that 
$\varphi(a)r = r\varphi(a) = r\varphi(rar)r$.  Similarly, 
$\varphi(a)(\unit - r) = (\unit - r)\varphi(a)$ and  
$$
\varphi(a)(\unit - r) 
=  (\unit - r)\varphi( (\unit - r)a (\unit - r)) (\unit - r)
= (\unit - r)\varphi( 0\unit) (\unit - r) = f(0)(\unit - r) 
$$
Hence $\varphi (a)\xi = f(0)(\unit - r)\xi = f(0)\xi $.  

We shall show that  for any $n\in \mathbb{N}$, $t_i\in [0,\infty)$ and projections 
$p_i \in M$ $(i=1,2,\ldots,n)$ with $\sum_{i=1}^n p_i= \unit$,
\[  \varphi(\sum_{i=1}^n t_ip_i) = \sum_{i=1}^n f(t_i)p_i.  \]
In fact, put $a = \sum_{i=1}^n t_ip_i$ and take any $k=1,2,\ldots,n$ and fix it. 
Put $b = t_k\unit$.  Then $ap_k = p_ka $ and $bp_k = p_kb$. Therefore 
$$
p_k\varphi(a) = \varphi(a) p_k = p_k\varphi(p_kap_k)p_k = p_k\varphi(t_kp_k)p_k
$$
and 
$$
f(t_k)p_k = \varphi(b) p_k= p_k\varphi(p_kbp_k)p_k = p_k\varphi(t_kp_k)p_k.
$$
Hence 
\[  \varphi(\sum_{i=1}^n t_ip_i) = \sum_{i=1}^n f(t_i)p_i.  \]

Next we shall show that, for any {\it invertible} $a \in M^+$ and 
any sequence $(a_n)_n $ 
in $M^+$ with $a_n \leq a$ , if  $\|a_n - a \| \rightarrow 0$, then 
$\|\varphi(a_n) - \varphi(a)\|  \rightarrow 0$.  In fact,  let 
$$
c_n = a^{-1} \# a_n := a^{-1/2}(a^{1/2}a_na^{1/2})^{1/2}a^{-1/2}
$$
be the geometric operator mean of $a^{-1}$ and $ a_n$, see??.   
Then $a_n = c_nac_n$ and  
$$
0 \leq c_n \leq a^{-1/2}(a^{1/2}aa^{1/2})^{1/2}a^{-1/2} = \unit.
$$
Because  $\|a_n - a \| \rightarrow 0$, we have that $\|c_n - \unit \| \rightarrow 0$.  
Since $\varphi$ is monotone and supercongruent, 
$$
0 \leq c_n \varphi(a)c_n \leq \varphi(c_n ac_n) = \varphi(a_n) \leq \varphi(a).
$$
Then  $\varphi(a) - \varphi(a_n) \leq \varphi(a) - c_n \varphi(a)c_n$. Hence 
$$
\| \varphi(a) - \varphi(a_n) \| \leq  \|\varphi(a) - c_n \varphi(a)c_n\| \rightarrow 0. 
$$
We shall also show that, for any {\it invertible} 
$a \in M^+$ and 
any sequence $(b_n)_n $ 
in $M^+$ with $a\leq b_n$ , if  $\|b_n - a \| \rightarrow 0$, then 
$\|\varphi(b_n) - \varphi(a) \|  \rightarrow 0$.  In fact,  let 
$$
d_n = a \# b_n^{-1} := a^{1/2}(a^{-1/2}b_n^{-1}a^{-1/2})^{1/2}a^{1/2}
$$
be the geometric operator mean of $a$ and $b_n^{-1}$.   
Then $a= d_nb_nd_n$ and  
$$
0 \leq d_n \leq  = a^{1/2}(a^{-1/2}a^{-1}a^{-1/2})^{1/2}a^{1/2} =\unit.
$$
Because  $\|b_n - a \| \rightarrow 0$, we have that $\|d_n - \unit \| \rightarrow 0$.  
Since $\varphi$ is monotone and supercongruent, 
$$
0 \leq d_n\varphi(b_n)d_n \leq \varphi (d_n b_nd_n)= \varphi(a), 
$$
and 
$0 \leq \varphi(b_n) \leq d_n^{-1}\varphi(a)d_n^{-1}$.  
Then  
$$\varphi(b_n) - \varphi(a) \leq d_n^{-1}\varphi(a)d_n^{-1} - \varphi(a)
$$
Hence 
$\| \varphi(b_n) - \varphi(a) \| \leq  \| d_n^{-1}\varphi(a)d_n^{-1} - \varphi(a)
 \| \rightarrow 0$.

In particular,  Since $\varphi(t\unit) = f(t)\unit$,  
the function $f$ is continuous on $(0,\infty)$.  
 Moreover 
$f$ is a Borel function on $[0,\infty)$.

For any {\it invertible} elememt $a \in M^+$. we shall show  that $\varphi(a)$ is 
equal to the continuous functional calculus of $a$ by $f|_{(0,\infty)}$ on $(0,\infty)$, that is, $\varphi(a)=f(a)$.  
We may assume that $\sigma(a)\subset [\alpha, \beta]$  for some 
$0 <\alpha \leq \beta$ in $(0,\infty)$.  
For any positive integer $n$, we define a function $g_n$ on $[\alpha, \beta]$ as follows:
\[   g_n(t) = \begin{cases} \alpha, & \alpha \le t \le \alpha + \dfrac{\beta-\alpha}{2^n} \\
              \alpha +(k-1)\dfrac{\beta-\alpha}{2^n}, &  \alpha +(k-1)\dfrac{\beta-\alpha}{2^n}<t 
              \le \alpha +k\dfrac{\beta-\alpha}{2^n} \end{cases} ,\]
where $k=2,3,\ldots,2^n$. 
Put $a_n =g_n(a)$.
Then we have $0 \leq a_n \le a$,  $\sigma(a_n)\subset [\alpha, \beta]$ and 
$\varphi(a_n) =f(a_n)$,  because $a_n$ has a finite spectra.  
Since $\|a_n - a \| \rightarrow 0$,  $\|\varphi(a_n) - \varphi(a)\|  \rightarrow 0$. 
On the otherhand, since $f$ is continuous on $(0,\infty)$ and 
the continuous functional calculus by $f$ on $[\alpha, \beta]$ is norm continuous, 
 $\|f(a_n) - f(a)\|\rightarrow 0$.  Therefore $\varphi(a)=f(a)$.

Because $M$ is an infinite-dimensional factor,  $M$ contains any finite matrix algebra 
$M_n({\mathbb C})$. Hence $f$ is an operator monotone 
continuous function on $(0,\infty)$. 

For possiblly non-invertible  element $a \in M^+$ in general, 
we shall show  that $\varphi(a)$ is equal to the Borel functional calculus of $a$ 
by $f$, that is $\varphi(a)=f(a)$.  This case is a little bit subtle.  
We may assume that $\sigma(a)\subset [0, \beta]$ fo some $\beta \geq 0$.

For any positive integer $n$, we define a function $\tilde{g_n}$ on $[0,\beta]$ as follows:
\[   \tilde{g_n}(t) = \begin{cases} 0, & 0 \le t \le \dfrac{\beta}{2^n} \\
              \dfrac{(k-1)\beta}{2^n}, &  \dfrac{(k-1)\beta}{2^n}<t 
              \le \dfrac{k\beta}{2^n} \end{cases} , \]
where $k=2,3,\ldots,2^n$. 
Put $a_n =\tilde{g_n}(a)$.
If $m\le n$ we have $a_m\le a_n\le a$. Since $\varphi$ is monotone, 
 $\varphi(a_m) \le \varphi(a_n) \le \varphi(a)$.  And 
 $\varphi(a_n) =f(a_n)$,  because $a_n$ has a finite spectra.
Since the sequence $\{ \tilde{g}_{n}\}$ converges the identity map on $[0, \beta]$ with respect to the pointwise convergent topology,  the increasing sequence 
$(a_n )_n = (\{ \tilde{g}_{n}\}(a))_n$  in $M^+$
 converges to $a$ in the strong operator topology. 

We do not know that $\varphi$ is normal  in this moment.  
But, only for this particular sequence $(a_n )_n $, 
 we can show that $\varphi(a_n)$ converges $\varphi(a)$ in the weak operator topology.  
 In fact, 
let $\tilde{h_n}$ be a bounded Borel function on $[0,\beta]$ as follows:
\[   \tilde{h_n}(t) =\begin{cases} 0, & 0\le t \le \dfrac{\beta}{2^n} \\
                      \sqrt{ \dfrac{\tilde{g_n}(t)}{t} } & \dfrac{\beta}{2^n}<t \le \beta \end{cases}.  \] 
Then $0\le \tilde{h}_{n}\le 1$ and $\{\tilde{h}_{n}\}$ poitwise converges to 
$\chi_{(0,\beta]}$.  We set $c_n = \tilde{h_n}(a)$.  
Then  the sequence $(c_{n} )_n$ of positive contractions strongly converges to the 
range projection $r = \chi_{(0,\beta]}(a)$ of $a$ and
$c_{n}ac_{n}= a_{n}$.

For any $\xi \in H$,  put $\xi_1 := r\xi$ and $\xi_2 := (\unit - r)\xi \in {\rm Ker }  a$.   
Since $a_n \leq a$,  ${\rm Ker }  a \subset {\rm Ker }  a_n$ and $\xi_2 \in {\rm Ker }  a_n$. 
Because  $ar = ra$ and $a_nr = ra_n$,  $\varphi(a)r = r\varphi(a)$ and 
$\varphi(a_n)r = r\varphi(a_n)$ .   
Since $\varphi$ is monotone and supercongruent, 
$c_n \varphi(a)c_n \le \varphi(c_nac_n) = \varphi(a_n) \le \varphi(a)$.  
 Thus we have
\[  0\le \varphi(a)-\varphi(a_n) \le  \varphi(a) -c_n\varphi(a)c_n . \]
Then  we have 
\begin{align*}
    0 & \leq <(\varphi(a) - \varphi(a_n))\xi, \xi> \\
    =  & <(\varphi(a) - \varphi(a_n))\xi_1, \xi_1> 
    + <(\varphi(a) - \varphi(a_n))\xi_2, \xi_2>\\
    \leq & <(\varphi(a) - c_n \varphi(a)c_n)\xi_1, \xi_1> 
    + <f(0)\xi_2, \xi_2>- <f(0)\xi_2, \xi_2> \\
    = &  <\varphi(a)\xi_1, \xi_1> - <\varphi(a_n)c_n\xi_1, c_n\xi_1>.  
\end{align*}

Since $c_n\xi_1$ convergents to $r\xi_1 = \xi_1$,  we conclude that 
$\varphi(a_n)$ converges $\varphi(a)$ in the weak operator topology.

We should note that a Borel functional calculus is not normal in general.  But 
we shall show that the  Borel functional calculus  $\varphi_f$ on $M^+$ by the particular function $f$ is normal.  In fact,  define a continuous function 
$F:[0,\infty)\longrightarrow [0,\infty)$ by 
\[   F(t) = f(t) - \lim_{t\to 0+} f(t).  \]
Then $F$ is operator monotone on $[0,\infty)$.  In fact, 
for $0\le a \le b$ and any $\epsilon>0$, $F(a+\epsilon \unit)\le F(b+\epsilon \unit)$ 
because $f$ is operator monotone on $(0,\infty)$. 
By the continuity of $F$, we can get $F(a)\le F(b)$ by making $\epsilon$ tend to $0$. 
Thus $F$ is operator monotone function on $[0,\infty)$ with $F(0)=0$. 
The functional calculus $\varphi_F$ by the continuous function $F$  is normal.
The function $f$ is decomposed into
\[   f(t) = F(t) + k \chi_{(0,\infty)}(t), \qquad k =\lim_{t\to 0+}f(t)-f(0)\ge 0. \]
Then the Borel functional calculus $\varphi_f$ of $a\in \M^+$ by $f$ has the form:
\[   \varphi_f(a) = \varphi_F(a) + k \varphi_{(0,\infty)}(a),  \]
where $\varphi_{(0,\infty)}(a)$ is the Borel functional calculus of $a$ 
by $\chi_{(0,\infty)}$ and 
in fact the range projection of $a$.  Hence $\varphi_{(0,\infty)}$ is normal by Propositon \ref{rangeprojection}.  
Therefore the Borel functional calculus $\varphi_f$  by $f$ is normal.

Finally, since $\varphi(a_n)$ converges $\varphi(a)$  and 
$f(a_n)$ converges $f(a)$ in the weak operator topology and 
$\varphi(a_n) =f(a_n)$, we conclude that $\varphi(a) = f(a)=\varphi_f(a)$, the 
Borel functional calculus of $a$ by $f$. \\
(2) $\Rightarrow$ (1): Suppose that 
there exists a Borel function $f:[0,\infty) \longrightarrow [0,\infty)$ such that 
$f$ is continuous on $(0,\infty)$ , operator monotone on $(0,\infty)$ with
\[   f(0) \le \lim_{t\to 0+}f(t),   \]
and  $\varphi(a)$ is equal to the Borel functional calculus 
$f(a)= \varphi_f(a)$ of $a$ by $f$ 
for any $a\in M^+$. 
We define a continuous function $F:[0,\infty)\longrightarrow [0,\infty)$ by 
\[   F(t) = f(t) - \lim_{t\to 0+} f(t).  \]
Then as in the preceding discussion, $F$ is operator monotone on $[0,\infty)$ 
with $F(0)=0$.  Hence 
the continuous functional calculus $\varphi_F$ is supercongruent as in Example \ref{monotonesupercongruent}.
The function $f$ is decomposed into
\[   f(t) = F(t) + k \chi_{(0,\infty)}(t), \qquad k =\lim_{t\to 0+}f(t)-f(0)\ge 0, \]
and 
\[   \varphi_f(a) = \varphi_F(a) + k \varphi_{(0,\infty)}(a),  \]
where $\varphi_{(0,\infty)}(a)$ is the range projection of $a$ and 
$\varphi_{(0,\infty)}$ is supercongruent.  Hence 
$\varphi$ is monotone and supercongruent. 

\end{proof}

By the above theorem, the restricted norm continuity and the normality of the non-linear 
positive map are satisfied automatically without assuming  them apriori. 
%
%
\begin{corollary} Let $M$ be a infinite-dimensional factor and 
$\varphi : M^+ \longrightarrow M^+$ be a non-linear positive map. 
If $\varphi$ is monotone and supercongruent,
then $\varphi$ is normal on $M^+
$ and 
$\varphi$ is norm continuous on the set of positive invertible elements 
$(M^+)^{-1}$.  Moreover $\varphi$ is concave. 
\end{corollary}
\begin{proof}
The almost all except concavity are proved in the discussion  of the proof in the 
theorem above.  Since $f_n(t) = t^{1/n}$ is a operator concave function on 
$[0,\infty)$ and $\chi_{(0,\infty)}(t) =  \lim_{n \to \infty} f_n(t)$,  
$\chi_{(0,\infty)}$ is also operator concave function.  Since $F$ is 
operator monotone on $[0,\infty)$, $F$ is also operator concave. 
Therfore $\varphi$ is concave. 
\end{proof}

In the above Theorem, if we weaken the supercongruent condition as 
only for 
positive {\it invertible} contraction $c \in M^+$  and $a \in M^+$
$$
c\varphi(a)c \le \varphi(cac), 
$$
then the conclusion of the Theorem above does not hold in general.  
In fact, let the function $f_\alpha$ $(\alpha\ge 0)$ be operator monotone on $[0,\infty)$ and increasing for $\alpha$, that is
\begin{gather*}
  a, b \in M^+ \text{ with } a\le b \Rightarrow f_{\alpha}(a)\le f_{\alpha}(b) \\
  \text{ and } \alpha \le \beta \Rightarrow f_{\alpha}(t) \le f_{\beta}(a) \quad (t\in[0,\infty)),  \tag{*}
\end{gather*}
for any factor $M$.
For an example, it is well-known 
\[   f(t) = \alpha -\frac{1}{t+1} \quad (\alpha\ge 0) \]
is operator monotone for $[0,\infty)$ (\cite{bhatia1}, \cite{bhatia2}, \cite{hiaipetz}).
So the function
\[   f_{\alpha}(t) = \frac{\alpha}{\alpha+1}-\frac{1}{t+1} \quad (\alpha\ge 0)  \]
satisfies the condition (*).

\begin{proposition}
We assume that $M = B(H)$ for a separable Hilbert space $H$ 
and the operator monotone function $f_{\alpha}$ on $[0,\infty)$ with the property (*) and
\[   f_{\infty}(t) = \lim_{\alpha \to \infty}f_{\alpha}(t) <\infty  \]
exists for all $t\in[0,\infty)$.
We define the map $\varphi : M^+\longrightarrow M^+$ as follows:
\[   \varphi(a) = f_{{\rm rank}(a)}(a)  \qquad a \in M^+,  \]
where ${\rm rank}(a) = \dim($the closure of $a\hilb)$.
Then we have the following. 
\begin{enumerate}
  \item[$(1)$] $a, b \in M^+$ $\Rightarrow$ $\varphi(a)\le \varphi(b)$.
  \item[$(2)$] For any invertible $c\in M$, $c^{*}\varphi(a)c \le \varphi (c^{*}ac)$   $(a\in M^+)$. 
  \item[$(3)$] If $f_m\neq f_n$ for some $m,n\in \mathbb{N}$, then $\varphi$ is not given as the continuous function calculus. 
\end{enumerate}
\end{proposition}
\begin{proof}
(1) Since $a\le b$, ${\rm rank}a \le {\rm rank} b$. So we have
\[  \varphi(a) = f_{{\rm rank}(a)}(a) \le f_{{\rm rank}(a)}(b) \le f_{{\rm rank}(b)}(b) = \varphi(b).  \]

(2) Since the mapping $f_{{\rm rank}(a)}$ is operator monotone on $[0,\infty)$, we have
\[  c^*f_{{\rm rank}(a)}(a)c = f_{{\rm rank}(a)}(c^*ac) , \]
using the approximation of polynomials for $f_{{\rm rank}(a)}$.
By the invertibility of $c$, we have ${\rm rank}(c^*ac) ={\rm rank}(a)$ and
\[  c^* \varphi(a) c = \varphi(c^*ac).  \]

(3) By definiton, we have $\varphi(t\unit) = f_{\infty}(t) \unit$ for any $t\in [0,\infty)$.
We assume $m<n$ and $f_m(t_0)<f_n(t_0)$ for some $t_0\in (0,\infty)$.
For a projection $p\in M$ with ${\rm rank}(p) = m$, we have
\[   \varphi(t_0 p) = f_m(t_0 p) <f_\infty(t_0 p) .  \]
\end{proof}


Finally we shall discuss  the ambiguity of operator means for non-invertible 
positive operators related with our Theorem , if we do {\it not} assume the 
upper semi-continuity for operator means.  We follow the original paper 
of Kubo-Ando \cite {kuboando} , see also  \cite{bhatia2}, \cite{hiaipetz}.

\begin{corollary}
Let $M$ be an infinite-dimensional factor.
If the mapping  
\[ \sigma: M^+\times M^+\ni (a,b) \mapsto a\sigma b \in M^+\]
satisfies the following conditions:
\begin{enumerate}
  \item[$(1)$] $a\le c$ and $b \le d$ imply $a \sigma b \le c\sigma d$.
  \item[$(2)$] For any $c\in M^+$, $c(a\sigma b)c \le (cac)\sigma (cbc)$.
\end{enumerate}
then there exist non-negative real valued, increasing, continuous functions $f$ and $g$ on $(0,\infty)$ such that
\begin{align*}
    a\sigma b & = b^{1/2}f(b^{-1/2}ab^{-1/2})b^{1/2} \\
                   & = a^{1/2}g(A^{-1/2}ba^{-1/2})a^{1/2}  
\end{align*}
for any positive invertible operators $a,b \in M^+$.
But we do not know how to represent $a\sigma b$ for 
positive non-invertible operators $a,b \in M^+$.
\end{corollary}
\begin{proof}
We define the mapping $\varphi: M^+ \longrightarrow M^+$ as follows:
\[   \varphi(a) = a\sigma \unit  \qquad (a\in M^+).  \]
It is clear that
\[   a\le b \; \Rightarrow \varphi(a)=a\sigma \unit \le b\sigma \unit = \varphi(b) \]
and for any contraction $c\in M^+$,
\[   cf(a)c = c(a\sigma \unit)c \le (cac)\sigma (c^2) \le (cac)\sigma \unit = f(cac). \]
By Theorem \ref{Borelcharacterization}, we can get the desired function $f$ and the relation
\[   f(a) = a\sigma \unit  \qquad (a\in (M^+)^{-1}).  \]

For any positive invertible operators $a, b\in M^+$, 
we have 
$$
a\sigma b = b^{1/2}f(b^{-1/2}ab^{-1/2})b^{1/2}
$$
as usual way:
\begin{align*}
   a \sigma b & = b^{1/2}b^{-1/2}(a\sigma b b^{-1/2} b^{1/2}  \\
      &  \le b^{1/2}( (b^{-1/2}ab^{-1/2}) \sigma \unit )b^{1/2} = b^{1/2}f(b^{-1/2}ab^{-1/2})b^{1/2} \\
      &  \le a\sigma b.
\end{align*}

We also define $\psi: M^+ \longrightarrow M^*$ as follows:
\[   \psi(a) = \unit \sigma a  \qquad (a\in M^+).  \]
Then we can get the function $g$ satisfying
\[   a\sigma b = a^{1/2}g(a^{-1/2}ba^{-1/2})a^{1/2} \quad \text{ for any } 
a,b \in  (M^+)^{-1}).  \]
\end{proof}

\begin{remark}  \rm
We do not know how to represent $a\sigma b$  for 
positive non-invertible operators $a,b \in M^+$.  
Based on the theory of Grassmann manifolds, Bonnabel-Sepulchre  \cite{B-S} 
and Batzies-H\''{u}per-Machado-Leite \cite{B-H-M-L} introduced  the 
geometric mean for positive semidefinite matrices or projections of fixed rank. 
Fujii \cite{F} extends it to a general theory of means of positive semideinite 
matrices of fixed rank. 
\end{remark}

\section{Non-additive measures and non-linear monotone positive maps}
In this section we begin to study non-linear monotone positive maps related 
with non-additive measures.  A non-additive measure is also called capacity, 
fuzzy measure, submeasure, monotone measure, etc. in different fields.  
Non-additive measures were firstly studied  by Choquet \cite{Ch} and 
Sugeno \cite{Su} . 
They proposed Choquet integral and Sugeno integral with respect to 
monotone measures.  

\begin{definition} \rm
 Let $\Omega$ be a set and ${\mathcal B}$ a $\sigma$-field on 
$\Omega$. A function $\mu: {\mathcal B} \rightarrow [0, \infty]$ is  called a 
monotone measure if $\mu$ satisfies 
\begin{enumerate}
\item[$(1)$]  $\mu(\emptyset) = 0$, and 
\item[$(2)$] For any $A,B \in {\mathcal B}$, if $A \subset B$, 
then $\mu(A) \leq \mu(B)$. 
\end{enumerate}
\end{definition}

We recall  the discrete Choquet integral with respect to a monotone measure on a finite set 
$\Omega  = \{1,2, \dots, n\}$.  Let ${\mathcal B} = P(\Omega)$ be the set of all 
subsets of $\Omega$ and $\mu:  {\mathcal B} \rightarrow [0, \infty)$ be a finite 
monotone measure . 
\begin{definition} \rm 
The  discrete Choquet integral  of $f = (x_1, x_2,\dots, x_n) \in [0,\infty)^n$ 
with respect to a monotone measure $\mu$ on a finite set 
$\Omega  = \{1,2, \dots, n\}$ is defined as follows:  
$$
(C)\int f d\mu = \sum_{i=1}^{n-1} (x_{\sigma(i) }- x_{\sigma(i+1)})\mu(A_i) 
+ x_{\sigma(n) }\mu(A_n) , 
$$
where $\sigma$ is a permutaion on $\Omega$ such that  
$x_{\sigma(1)} \geq x_{\sigma(2)} \geq  \dots \geq x_{\sigma(n)}$ 
, $A_i = \{\sigma(1),\sigma(2),\dots,\sigma(i)\}$.  Here we should note that 
$$
f = \sum_{i=1}^{n-1} (x_{\sigma(i) }- x_{\sigma(i+1)})\chi_{A_i} 
+ x_{\sigma(n) }\chi_{A_n}
$$
\end{definition}
Let $A = {\mathbb C}^n$ and define 
$(C-\varphi)_{\mu} :( {\mathbb C}^n)^+ \rightarrow {\mathbb C}^+$ 
by the  Choquet integral $(C-\varphi)_{\mu}(f) = (C)\int f d\mu$.  Then 
$(C-\varphi)_{\mu}$ is a non-linear monotone positive map such that 
$(C-\varphi)_{\mu}(\alpha f) = \alpha(C- \varphi)_{\mu}(f)$ for a positive 
scalar $\alpha$.

We shall consider a matrix version of the discrete Choquet integral. 
\begin{proposition} Let $\mu:  {\mathcal B} \rightarrow [0, \infty)$ be a finite 
monotone measure on a finite set 
$\Omega  = \{1,2, \dots, n\}$  with ${\mathcal B} = P(\Omega)$. 
Let $A = M_n({\mathbb C})$ and define 
$(C-\varphi)_{\mu} : (M_n({\mathbb C}))^+ \rightarrow  {\mathbb C}^+$ 
as follows:  For $a \in (M_n({\mathbb C}))^+$, let 
$\lambda(a) = (\lambda_1(a),\lambda_2(a),\dots, \lambda_n(a))$ be the list of the 
eigenvalues of $a$ in decreasing order :
$\lambda_1(a), \geq \lambda_2(a) \geq \dots \geq \lambda_n(a)$ with 
counting multiplicities.  Let
$$
(C-\varphi)_{\mu}(a) = \sum_{i=1} ^{n-1} ( \lambda_i(a)- \lambda_{i+1}(a))\mu(A_i) 
+ \lambda_n (a) \mu(A_n) , 
$$
where $A_i = \{1,2,\dots,i \}$.
Then $(C-\varphi)_{\mu}$ is a unitary invariant  non-linear monotone positive map such that 
$(C-\varphi)_{\mu}(\alpha a) = \alpha (C-\varphi)_{\mu}(a)$ for a positive 
scalar $\alpha$. 
\end{proposition}
\begin{proof} For $a, b  \in (M_n({\mathbb C}))^+$, suppose  that $0 \leq a \leq b$. 
By the mini-max  principle for eigenvalues, we have that 
$\lambda_i(a) \leq \lambda_i(b)$ for $i = 1,2,\dots,n$.  
\begin{align*}
(C-\varphi)_{\mu}(a) & = \sum_{i=1} ^{n-1} ( \lambda_i (a) - \lambda_{i+1} (a))\mu(A_i) 
+ \lambda_n (a) \mu(A_n)  \\
      & = \sum_{i=2} ^n  \lambda_i (a)(\mu(A_i) - \mu(A_{i-1}) ) +
        \lambda_1 (a)(\mu(A_1) \\
  & \leq \sum_{i=2} ^n  \lambda_i (b)(\mu(A_i) - \mu(A_{i-1}) ) + 
  \lambda_1 (b)\mu(A_1) \\
  &= (C-\varphi)_{\mu} (b)
\end{align*}
since $\mu$ is a monotone measure. Thus $\varphi_{\mu}$ is monotone. It is clear 
that $\varphi_{\mu} (\alpha a) = \alpha \varphi_{\mu} (a)$ for a positive 
scalar $\alpha$ by the definiton and $\varphi_{\mu}$ is a unitary invariant.
\end{proof}　
Furthermore we can replace a monotone measure on a finite set 
$\Omega  = \{1,2, \dots, n\}$ by a positive operator-valued monotone measure  
$\mu:  {\mathcal B} \rightarrow B(H)^+$ for some Hilbert space $H$ , that is, 
\begin{enumerate}
\item[$(1)$]  $\mu(\emptyset) = 0$, and 
\item[$(2)$] For any $X,Y \in {\mathcal B} = P(\Omega)$, if $X \subset Y$, 
then $\mu(X) \leq \mu(Y)$. 
\end{enumerate}
We have a similar result as follows: 
\begin{proposition}  Let $H$ be a Hilbert space , 
$\mu:  {\mathcal B} \rightarrow B(H)^+$ be a 
positive operator-valued  monotone measure on a finite set 
$\Omega  = \{1,2, \dots, n\}$  with ${\mathcal B} = P(\Omega)$. 
Define 
$(C-\varphi)_{\mu} : (M_n({\mathbb C}))^+ \rightarrow  B(H)^+$ 
as follows:  For $a \in (M_n({\mathbb C}))^+$, let 
$\lambda(a) = (\lambda_1(a),\lambda_2(a),\dots, \lambda_n(a))$ be the list of the 
eigenvalues of $a$ in decreasing order with counting multiplicities.  Let
$$
(C-\varphi)_{\mu}(a) = \sum_{i=1} ^{ n-1} ( \lambda_i(a)- \lambda_{i+1}(a))\mu(A_i) 
+ \lambda_n(a) \mu(A_n) , 
$$
where $A_i = \{1,2,\dots,i\}$.
Then $(C-\varphi)_{\mu} $ is a a unitary invariant non-linear 
monotone positive map such that 
$(C-\varphi)_{\mu}(\alpha a) = \alpha (C-\varphi)_{\mu}(a)$ for a positive 
scalar $\alpha$. 
\end{proposition}
\begin{proof}　Use the similar argument as above.
\end{proof}　

Honda and Okazaki \cite{H-O} proposed the inclusion-exclusion integral with respect to a 
monotone measure, which is a 
generalization of the Lebesgue integral and the the Choquet integral. 
We can also consider a matrix version of the inclusion-exclusion integral. 
\begin{proposition}  Let $H$ be a Hilbert space , 
$\mu:  {\mathcal B} \rightarrow B(H)^+$ be a 
positive operator-valued  monotone measure on a finite set 
$\Omega  = \{1,2, \dots, n\}$  with ${\mathcal B} = P(\Omega)$. 
Fix a positive number $K$. Let $(\Omega, P(\Omega),\mu, I,K)$ be an 
interactive monotone measure space such that the interaction operator 
$I$ is positive and monotone in the sense of \cite{H-O}. 
Define 
$$
(I-\varphi)_{\mu} : \{a \in (M_n({\mathbb C}))^+ \ | 
\ \sigma (a) \subset [0,K] \}  \rightarrow  B(H)^+
$$ 
as follows:  For $a \in (M_n({\mathbb C}))^+$ with the spectrum 
$\sigma (a) \subset [0,K]$ , let 
$\lambda(a) = (\lambda_1(a),\lambda_2(a),\dots, \lambda_n(a))$ be the list of the 
eigenvalues of $a$ in decreasing order with counting multiplicities.  Let
$$
(I-\varphi)_{\mu}(a) = \sum_{A \in P(\Omega) } 
(\sum_{B\supset A} (-1)^{|B\setminus A|}I(\lambda(a)|B))\mu(A). 
$$
Then $(I-\varphi)_{\mu} $ is a a unitary invariant non-linear 
monotone positive map. 
\end{proposition}
\begin{proof}　For $a, b  \in (M_n({\mathbb C}))^+$ 
with the spectra $\sigma (a) \subset [0,K]$ and 
$\sigma (b) \subset [0,K]$, suppose  that $0 \leq a \leq b$. 
By the mini-max  principle for eigenvalues, we have that 
$\lambda_i(a) \leq \lambda_i(b)$ for $i = 1,2,\dots,n$.  
Since the interaction operator $I$ is monotone, 
$$
\sum_{B\supset A} (-1)^{|B\setminus A|}I(\lambda(a)|B)
\leq \sum_{B\supset A} (-1)^{|B\setminus A|}I(\lambda(b)|B).
$$
Therefore $(I-\varphi)_{\mu}(a)  \leq  (I-\varphi)_{\mu}(b)$. 
\end{proof}

Next we recall  the Sugeno integral with respect to a monotone measure on a finite set 
$\Omega  = \{1,2, \dots, n\}$.  
\begin{definition} \rm 
The  discrete Sugeno integral  of $f = (x_1, x_2,\dots, x_n) \in [0,\infty)^n$ 
with respect to a monotone measure $\mu$ on a finite set 
$\Omega  = \{1,2, \dots, n\}$ is defined as follows:  
$$
(S)\int f d\mu = \vee_{i=1}^{n} (x_{\sigma(i) } \wedge \mu(A_i) ) , 
$$
where $\sigma$ is a permutaion on $\Omega$ such that  
$x_{\sigma(1)} \geq x_{\sigma(2)} \geq  \dots \geq x_{\sigma(n)}$ 
, $A_i = \{\sigma(1),\sigma(2),\dots,\sigma(i)\}$ and 
$\vee =\max$ , $\wedge = \min$. 
Here we should note that 
$$
 f = \vee_{i=1}^{n} (x_{\sigma(i) } \chi_{A_i}) 
$$
\end{definition}
Let $A = {\mathbb C}^n$ and define 
$(S-\varphi)_{\mu} :( {\mathbb C}^n)^+ \rightarrow {\mathbb C}^+$ 
by the  Sugeno integral $(S-\varphi)_{\mu}(f) = (S)\int f d\mu$.  Then 
$(S-\varphi)_{\mu}$ is a non-linear monotone positive map such that 
$(S-\varphi)_{\mu}(\alpha f) = \alpha(S- \varphi)_{\mu}(f)$ for a positive 
scalar $\alpha$. 

We shall consider a matrix version of the discrete Sugeno integral. 
\begin{proposition} Let $\mu:  {\mathcal B} \rightarrow [0, \infty)$ be a finite 
monotone measure on a finite set 
$\Omega  = \{1,2, \dots, n\}$  with ${\mathcal B} = P(\Omega)$. 
Let $A = M_n({\mathbb C})$ and define 
$(S-\varphi)_{\mu} : (M_n({\mathbb C}))^+ \rightarrow  {\mathbb C}^+$ 
as follows:  For $a \in (M_n({\mathbb C}))^+$, let 
$\lambda(a) = (\lambda_1(a),\lambda_2(a),\dots, \lambda_n(a))$ be the list of the 
eigenvalues of $a$ in decreasing order :
$\lambda_1(a), \geq \lambda_2(a) \geq \dots \geq \lambda_n(a)$ with 
counting multiplicities.  Let
$$
(S-\varphi)_{\mu}(a) = \vee_{i=1}^{n}  (\lambda_i(a) \wedge \mu(A_i) )
$$
where $A_i = \{1,2,\dots,i \}$.
Then $(S-\varphi)_{\mu}$ is a unitary invariant  
non-linear monotone positive map such that 
$(S-\varphi)_{\mu}(\alpha a) = \alpha (S-\varphi)_{\mu}(a)$ for a positive 
scalar $\alpha$. 
\end{proposition}
\begin{proof} For $a, b  \in (M_n({\mathbb C}))^+$, suppose  that $0 \leq a \leq b$. 
Since 
$\lambda_i(a) \leq \lambda_i(b)$ for $i = 1,2,\dots,n$, 
\begin{align*}
(S-\varphi)_{\mu}(a) & = \vee_{i=1}^{n} (\lambda_i(a) \wedge \mu(A_i) ) \\
     & \leq  \vee_{i=1}^{n} (\lambda_i(b) \wedge \mu(A_i) )  
  = (S-\varphi)_{\mu} (b). 
\end{align*}
Thus $\varphi_{\mu}$ is monotone. It is clear 
that $\varphi_{\mu} (\alpha a) = \alpha \varphi_{\mu} (a)$ for a positive 
scalar $\alpha$ by the definiton and $\varphi_{\mu}$ is a unitary invariant.
\end{proof}　

\end{document}